\documentclass[12pt]{amsart}
\usepackage{amsfonts,amssymb,amscd,amsmath,enumerate,color}
\def\NZQ{\mathbb}               % the font for N,Z,Q,R,C

\def\ZZ{{\NZQ Z}}
\def\RR{{\NZQ R}}

\def\frk{\mathfrak}               % font for "Fraktur"

\def\Phi{{\frk N}}
\def\eb{{\bold e}}
\def\sb{{\bold s}}
\def\tb{{\bold t}}
\def\xb{{\bold x}}
\def\yb{{\bold y}}
\def\zb{{\bold z}}
\def\opn#1#2{\def#1{\operatorname{#2}}} % to make operators
\opn\chara{char} 
\opn\length{\ell} 
\opn\pd{pd} 
\opn\rk{rk}
\opn\projdim{proj\,dim} 
\opn\injdim{inj\,dim} 
\opn\rank{rank}
\opn\depth{depth} 
\opn\grade{grade} 
\opn\height{height}
\opn\embdim{emb\,dim} 
\opn\codim{codim}

\opn\Tr{Tr} 
\opn\bigrank{big\,rank}
\opn\superheight{superheight}
\opn\lcm{lcm}
\opn\trdeg{tr\,deg}%\emph{
\opn\reg{reg} 
\opn\lreg{lreg} 
\opn\ini{in} 
\opn\lpd{lpd}
\opn\size{size}
\opn\mult{mult}
\opn\dist{dist}
\opn\cone{cone}
\opn\lex{lex}
\opn\rev{rev}
\opn\div{div} \opn\Div{Div} \opn\cl{cl} \opn\Cl{Cl}
\opn\Spec{Spec} \opn\Supp{Supp} \opn\supp{supp} \opn\Sing{Sing}
\opn\Ass{Ass} \opn\Min{Min}
\opn\Ann{Ann} \opn\Rad{Rad} \opn\Soc{Soc}
\opn\Syz{Syz} \opn\Im{Im} \opn\Ker{Ker} \opn\Coker{Coker}
\opn\Am{Am} \opn\Hom{Hom} \opn\Tor{Tor} \opn\Ext{Ext}
\opn\End{End} \opn\Aut{Aut} \opn\id{id} \opn\ini{in}

\opn\nat{nat}
\opn\pff{pf}%   \pf exists already
\opn\Pf{Pf} \opn\GL{GL} \opn\SL{SL} \opn\mod{mod} \opn\ord{ord}
\opn\Gin{Gin}
\opn\Hilb{Hilb}\opn\adeg{adeg}\opn\std{std}\opn\ip{infpt}
\opn\Pol{Pol}
\opn\sat{sat}
\opn\Var{Var}
\opn\Gen{Gen}

\opn\aff{aff} \opn\con{conv} \opn\relint{relint} \opn\st{st}
\opn\lk{lk} \opn\cn{cn} \opn\core{core} \opn\vol{vol}
\opn\link{link} \opn\star{star}
\opn\gr{gr}

\def\Hc{{\mathcal H}}

\def\Pc{{\mathcal P}}
\def\Qc{{\mathcal Q}}

\def\pot#1#2{#1[\kern-0.28ex[#2]\kern-0.28ex]}

\opn\dirlim{\underrightarrow{\lim}}
\opn\inivlim{\underleftarrow{\lim}}
\let\iso=\cong

\let\to=\rightarrow

\def\Implies{\ifmmode\Longrightarrow \else
        \unskip${}\Longrightarrow{}$\ignorespaces\fi}
\def\implies{\ifmmode\Rightarrow \else
        \unskip${}\Rightarrow{}$\ignorespaces\fi}
\def\iff{\ifmmode\Longleftrightarrow \else
        \unskip${}\Longleftrightarrow{}$\ignorespaces\fi}

\let\:=\colon
\newtheorem{Theorem}{Theorem}[section]
\newtheorem{Lemma}[Theorem]{Lemma}
\newtheorem{Corollary}[Theorem]{Corollary}
\newtheorem{Proposition}[Theorem]{Proposition}
\newtheorem{Remark}[Theorem]{Remark}

\newtheorem{Conjecture}[Theorem]{Conjecture}

\let\epsilon\varepsilon
\let\phi=\varphi
\let\kappa=\varkappa
\def\qed{\ifhmode\textqed\fi
      \ifmmode\ifinner\quad\qedsymbol\else\dispqed\fi\fi}
\def\textqed{\unskip\nobreak\penalty50
       \hskip2em\hbox{}\nobreak\hfil\qedsymbol
       \parfillskip=0pt \finalhyphendemerits=0}
\def\dispqed{\rlap{\qquad\qedsymbol}}

\opn\dis{dis}
\opn\height{height}
\opn\dist{dist}
\def\pnt{{\raise0.5mm\hbox{\large\bf.}}}

\opn\Lex{Lex}

\def\Pb{{\bold P}}
\def\Ib{{\bold I}}
\opn\Asc{Asc}
\opn\asc{asc}
\opn\Pyr{Pyr}
\opn\Ehr{Ehr}
\opn\Vol{Vol}
\opn\Chim{Chim}
\begin{document}
\title{Gorenstein properties and integer decomposition properties of lecture hall polytopes}
\author{Takayuki Hibi, McCabe Olsen, and Akiyoshi Tsuchiya}
\address{Takayuki Hibi,
Department of Pure and Applied Mathematics,
Graduate School of Information Science and Technology,
Osaka University, Suita, Osaka 565-0871, Japan}
\email{hibi@math.sci.osaka-u.ac.jp}
\address{McCabe Olsen,
Department of Mathematics,
University of Kentucky,
Lexington, KY 40506--0027, USA}
\email{mccabe.olsen@uky.edu}
\address{Akiyoshi Tsuchiya,
Department of Pure and Applied Mathematics,
Graduate School of Information Science and Technology,
Osaka University, Suita, Osaka 565-0871, Japan}
\email{a-tsuchiya@cr.math.sci.osaka-u.ac.jp}
\thanks{
The second author was supported by a 2016 NSF/JSPS EAPSI Fellowship award NSF OEIS-1613525. The third author is partially supported by Grant-in-Aid for JSPS Fellows 16J01549.}
\subjclass[2010]{05A20, 05E40, 13P20, 52B20}
\keywords{Lecture hall polytopes, Gorenstein polytopes, integer decomposition property
}
\begin{abstract}
Though much is known about $\sb$-lecture hall polytopes, there are still many unanswered questions.
In this paper, we show that $\sb$-lecture hall polytopes satisfy the integer decomposition property (IDP) in the case of monotonic $\sb$-sequences. Given restrictions on a monotonic $\sb$-sequence, we discuss necessary and sufficient conditions for the Fano, reflexive and Gorenstein  properties. Additionally, we give a construction for producing Gorenstein/IDP lecture hall polytopes.  
\end{abstract}
\maketitle

\section{Introduction}
Let $\Pc\subset\RR^d$ be a $d$-dimensional convex lattice polytope. For $t\in \ZZ_{>0}$, \emph{lattice point enumerator} $i(\Pc,t)$ gives the number of lattice points in $t\Pc=\{t \alpha:\alpha\in \Pc\}$, the  $t$th dilation of $\Pc$. That is,
	\[
	i(\Pc,t)=\#(t\Pc\cap \ZZ^d), \ \ t\in\ZZ_{>0}.
	\]
Provided that $\Pc$ is a lattice polytope, it is known  that this is a polynomial in the variable $t$ of degree $d$ (\cite{Ehrhart}).  
The \emph{Ehrhart series} for $\Pc$, $\Ehr_{\Pc}{(\lambda)}$, is the rational generating function 
	\[
	\Ehr_{\Pc}(\lambda)= \sum_{t\geq 0}i(\Pc,t)\lambda^t=\frac{\delta(\Pc,\lambda)}{(1-\lambda)^{d+1}}
	\]
where $\delta(\Pc,\lambda)=\delta_0+\delta_1\lambda+\delta_2\lambda^2+\cdots+\delta_d\lambda^d$ is the \emph{$\delta$-polynomial} of $\Pc$ and $\delta(\Pc)=(\delta_0,\delta_1,\delta_2,\ldots,\delta_d)$ the $\delta$-vector of $\Pc$. The $\delta$-polynomial ($\delta$-vector) is endowed with the following properties:
	\begin{itemize}
	\item $\delta_0=1$, $\delta_1=i(\Pc,1)-(d+1)$, and $\delta_d=\#(\Pc\setminus \partial\Pc\cap \ZZ^d)$;
	\item $\delta_i\geq 0$ for all $0\leq i\leq d$ (\cite{StanleyNonNeg});
	\item If $\delta_d\neq 0$, then $\delta_1\leq \delta_i$ for each $0\leq i\leq d-1$ (\cite{HibiBounds}). 
	\end{itemize}

The Ehrhart series and $\delta$-polynomials for polytopes have been studied extensively. For a detailed background on these topics, please refer to \cite{BeckRobins,Ehrhart,HibiRedBook,EC1}.

Let $\ZZ^{d \times d}$ denote the set of $d \times d$ integer matrices.
A matrix $A \in \ZZ^{d \times d}$ is {\em unimodular} if $\det (A) = \pm 1$.
Given lattice polytopes $\Pc \subset \RR^{d}$ 
and $\Qc \subset \RR^{d}$ of dimension $d$,
we say that $\Pc$ and $\Qc$ are {\em unimodularly equivalent}
if there exists a unimodular matrix $U \in \ZZ^{d \times d}$
and a vector ${\bf w} \in \ZZ^{d}$
such that $\Qc=f_U(\Pc)+{\bf w}$,
where $f_U$ is the linear transformation of $\RR^d$ defined by $U$,
i.e., $f_U({\bf v}) = {\bf v} \, U$ for all ${\bf v} \in \RR^d$.

We say that a lattice polytope $\Pc$ is \emph{Fano} if $(\Pc\setminus \partial\Pc) \cap \ZZ^d=\{{\bf 0}\}$. We say that $\Pc$ is \emph{reflexive} if it is Fano and its dual polytope
	\[
	\Pc^\vee=\left\{y\in\RR^d \ : \langle x,y \rangle\leq 1 \mbox{ for all } x\in\Pc \right\}
	\]
is a lattice polytope. 
Moreover, it follows from \cite{HibiDualPolytopes} that the following statements are equivalent:
	\begin{itemize}
	\item $\Pc$ is unimodularly equivalent to some reflexive polytope;
	\item  $\delta(\Pc,\lambda)$ is of degree $d$ and is symmetric, that is $\delta_i=\delta_{d-i}$ for $0\leq i \leq \lfloor \frac{d}{2} \rfloor$.
	\end{itemize}
We say that $\Pc$ is \emph{Gorenstein of index} $c$ where $c\in \ZZ_{>0}$ if $c\Pc$ is unimodularly equivalent to a reflexive polytope \cite{DeNegriHibi}.  Equivalently, $\Pc$ is Gorenstein if and only if $\delta(\Pc,\lambda)$ is symmetric with $\deg(\delta(\Pc,\lambda))=d-c+1$ 
(\cite{StanleyHFGA}).
	
We now give the definition of lecture hall polytopes. For a sequence of positive integers $\sb=(s_1,s_2,\ldots,s_d)$, the \emph{$\sb$-lecture hall polytope} is
	\[
	\Pb_d^{(\sb)}:=\left\{\xb\in \RR^d \ : \ 0\leq \frac{x_1}{s_1}\leq \frac{x_2}{s_2}\leq \cdots\leq \frac{x_d}{s_d}\leq 1 \right\}
	\] 
which alternatively has the vertex representation as the column vectors of the matrix
	\[
	\begin{bmatrix}
	0 & s_d & s_d & s_d & \cdots & s_d\\
	0 & 0 & s_{d-1} & s_{d-1} & \cdots & s_{d-1}\\
	0 & 0 & 0 & s_{d-2} & \cdots & s_{d-2}\\
	\vdots & \vdots & \vdots & & \ddots & \vdots\\
	0 & 0 & 0 & \cdots & 0 & s_1\\
	\end{bmatrix}
	\]
where $x_d$ is given by the first row and so on with $x_1$ given by the last row. It should be noted that there is a easy unimodular equivalence $\Pb_d^{(\sb)}\iso \Pb_d^{(s_d,\ldots,s_2,s_1)}$.

For a given $\sb=(s_1,s_2,\ldots, s_d)$, we define the \emph{$\sb$-inversion sequences} by the set  $\Ib_d^{(\sb)}:=\{\eb\in\ZZ^d \ : \ 0\leq e_i <s_i\}$. Given $\eb\in\Ib_d^{(\sb)}$, we define the \emph{ascent set} of $\eb$ by 
	\[
	 \Asc \, \eb:=\left\{ i \ : 0\leq i < d  \mbox{ and }  \frac{e_i}{s_i}<\frac{e_{i+1}}{s_{i+1}} \right\}
	\]	
with the convention that $e_0=1$ and $s_0=1$.
Let $\asc \, \eb:=|\Asc \, \eb|$. 
The following result of the $\delta$-polynomials of $\sb$-lecture hall polytopes for arbitrary $\sb$. 

\begin{Lemma}[{\cite[Theorem 5]{SavageSchuster}}]
	\label{deltapoly}
For any $\sb$, the $\delta$-polynomial of $\Pb_d^{(\sb)}$ is given by
	\[
	\delta\left(\Pb_d^{(\sb)},\lambda\right)=\sum_{\eb\in\Ib_d^{(\sb)}}\lambda^{\asc \, \eb}.
	\]
Moreover, these polynomials are real-rooted and hence unimodal.	
\end{Lemma}

The theory of lecture hall polytopes and lecture hall partitions is extensive \cite{LectureHallMath} and many questions have been answered. 
Some particular motivating work includes the thorough study of Gorenstein properties for {\it $\sb$-lecture hall cones} \cite{GorensteinCones}. These results imply Ehrhart theoretic properties of the {\it rational $\sb$-lecture hall polytopes} ${\bf R}_{(\sb)}^d$, but do not imply the same properties for $\Pb_{(\sb)}^d$. 
Additionally, the existence of a unimodular triangulation for the $\sb$-lecture hall cone of $\sb=(1,2,\cdots,d)$ was recently shown \cite{BBKSZ}.
 However, showing the existence or nonexistence of a unimodular triangulation of $\Pb_{(\sb)}^d$ for most $\sb$ is still an open question. 
This motivates the following unanswered questions:  

	\begin{itemize}
	\item For what $\sb$ is $\Pb_d^{(\sb)}$ Fano, reflexive, or Gorenstein?
	\item For what $\sb$ does $\Pb_d^{(\sb)}$  satisfies the integer decomposition property?
	\item If $\Pb_d^{(\sb)}$ satisfies the integer decomposition property, for what conditions will it admit a unimodular triangulation?
	\end{itemize}

In this paper, we answer these questions for particular large classes of $\sb$ as progress towards a complete characterization. 
First we consider  $\Pb_d^{(\sb)}$ when $\sb$ is a monotonic sequence. 
We will show necessary and sufficient conditions for Fano and reflexive in the case when $\sb$ is a sequence with $0\leq s_{i+1}-s_i\leq 1$ for all $0\leq i\leq d-1$ (or equivalently   $0\leq s_{i}-s_{i-1}\leq 1$ for all $0\leq i\leq d-1$), the case when $\sb$ is a strictly monotonic sequence, and the case when $\sb$ is constant then strictly increasing. 
In the two latter cases, we can also provided necessary and sufficient conditions for when $\Pb_d^{(\sb)}$ is Gorenstein. We continue to show that $\Pb_d^{(\sb)}$ satisfies the integer decomposition property for all monotonic $\sb$ and show that in some special cases, we can prove that $\Pb_d^{(\sb)}$ admits a unimodular triangulation, which is a stronger condition. Furthermore, if we have two lecture hall polytopes $\Pb_d^{(\sb)}$ and $\Pb_e^{(\tb)}$ which are Gorenstein and/or satisfies the integer decomposition property, we can construct a $(d+e+1)$-dimensional lecture hall polytope with the respective property.
 
\section{Fano, Reflexive, and Gorenstein}
Suppose that $\sb$ is a monotonic sequence. We give necessary and sufficient conditions for when $\Pb_d^{(\sb)}$ is Fano or reflexive in the special cases of  $\sb$ a strictly increasing sequence and $\sb$ a sequence which increases by at most one. In the case of strictly increasing, we can also find necessary and sufficient conditions for when $\Pb_d^{(\sb)}$is Gorenstein.

\begin{Remark}
All of the results in this section can be rephrased in the obvious way for when $\sb$ is decreasing. This follows from the observation $\Pb_d^{(s_1,s_2,\ldots,s_d)}\iso  \Pb_d^{(s_d,s_{d-1},\ldots,s_1)}$.
\end{Remark} 

\subsection{Strictly increasing $\sb$-sequences}
Suppose that $\sb=(s_1,s_2,\ldots,s_d)$ is a sequence of positive integers such that $s_i\lneq s_{i+1}$ for all $i\in\{1,2,\ldots,d-1\}$. We have the following necessary and sufficient conditions for when $\Pb_d^{(\sb)}$ is translation equivalent to a Fano polytope.

\begin{Theorem}\label{strictincreaseFano}
Suppose $\sb$ is a sequence of strictly increasing positive integers. 
Then $\Pb_d^{(\sb)}$ is translation equivalent to a Fano polytope if and only if $s_1=2$ and $s_{i+1}\leq 2s_{i}$ for all $1\leq i \leq d-1$. Moreover, if $\Pb_d^{(\sb)}$ is Fano, the unique interior point of $\Pb_d^{(\sb)}$ is $(s_d-1,s_{d-1}-1,\ldots,s_2-1,s_1-1)^T$.
\end{Theorem}

\begin{proof}
Suppose that $\sb$ is a sequence with the property that $s_1=2$ and $s_{i+1}\leq 2s_{i}$. We will show that this implies that $\Pb_d^{(\sb)}$ is Fano. It is sufficient to show that $\Ib_d^{(\sb)}$ has exactly 1 inversion sequence $\eb$ such that $\asc \,\eb=d$, as this implies that $\delta_d ( \Pb_d^{(\sb)})=1$ by Lemma \ref{deltapoly}.
If we let $\eb=(s_1-1,s_2-1,s_3-1,\ldots, s_d-1)$, we should note that $\asc\,\eb=d$ because 
	\[
	\frac{s_i-1}{s_i}<\frac{s_{i+1}-1}{s_{i+1}}
	\]
follows from the fact that $-s_{i+1}< -s_{i}$ which is true by assumption. To claim that this is the only such inversion sequence note that 
	\[
	\frac{s_i-1}{s_i}<\frac{s_{i+1}-2}{s_{i+1}}
	\]
is never true for any $i$ because this would imply that $-s_{i+1}<-2s_i$ which is false by assumption. Moreover, in order for $\eb$ to have an ascent in position 1, we need $e_1=1=s_1-1$, so it follows that there is a single inversion sequence of this type. Hence,  Additionally, we should note that because we have	
	\[
	0<\frac{s_1-1}{s_1}<\frac{s_2-1}{s_2}<\cdots<\frac{s_d-1}{s_d}<1
	\]
it follows that the point $(s_d-1,s_{d-1}-1,\ldots,s_2-1,s_1-1)^T$ does not lie on a supporting hyperplane and is hence the unique interior point of $\Pb_d^{(\sb)}$.

Now, suppose that $\sb$ is not of the prescribed form. We will show that $\Pb_d^{(\sb)}$ is not Fano. There are three possible cases:
	\begin{enumerate}[(i)]
	\item $s_1=1$;
	\item $s_1\geq 3$;
	\item $s_1=2$ and $s_{i+1}>2s_i$ for some $1\leq i\leq d-1$.
	\end{enumerate}
Each of these cases preclude $\Pb_d^{(\sb)}$ from being Fano.

For (i), if $s_1=1$, it is clear from the vertex description of the polytope that $\Pb_d^{(\sb)}\cong \Pyr(\Pb_{d-1}^{(s_2,s_3,\ldots,s_d)})$	and hence $\delta_d(\Pb_d^{(\sb)})=0$. 

For (ii), if $s_1\geq 3$, it is easy to see that $\Pb_d^{(3,4,\ldots,d+2)}\subseteq \Pb_d^{(\sb)}$. We can see that $\delta_d(\Pb_d^{(3,4,\ldots,d+2)})\geq 2$ because both the inversion sequences $\eb=(1,2,\ldots,d)$ and $\eb'=(2,3,\ldots, d+1)$ have the property $\asc\,\eb=\asc\,\eb'=d$. So, $\Pb_d^{(3,4,\ldots,d+2)}$ has at least 2 interior points, which must also be interior points of $\Pb_d^{(\sb)}$, meaning it is not Fano. 

For (iii), if we have $s_1=2$ but that there exists at least one $1\leq i\leq d-1$ such that $s_{i+1}>2s_{i}$. If there exist multiple such $i$, choose the smallest. We can see that $\Pb_d^{(\tb)}\subseteq\Pb_d^{(\sb)}$, where $\tb=(s_1,\ldots,s_i,2s_i+1,2s_i+2,\ldots,2s_i+(d-i+1))$. If we consider this smaller polytope, we can again ascertain that $\delta_d(\Pb_d^{(\tb)})\geq 2$. Note that $\eb=(s_1-1,\ldots s_i-1, 2s_i,2s_i+1,\ldots, 2s_i+(d-i))$ has $\asc\,\eb=d$ as 
	\[
	\frac{s_i-1}{s_i}< \frac{2s_i}{2s_i+1}
	\] 
follows from $-s_i-1<0$ and the other inequalities follow from previous arguments. However, $\eb'=(s_1-1,\ldots s_i-1, 2s_i-1,2s_i,\ldots, 2s_i+(d-i-1))$ also has the property $\asc\,\eb'=d$ because 
	\[
	\frac{s_i-1}{s_i}< \frac{2s_i-1}{2s_i+1}
	\]
is follows from $-1<0$ and 
	\[
	\frac{2s_i+k}{2s_i+k+2}<\frac{2s_i+k+1}{2s_i+k+3}
	\]	
follows from $0<4s_i+2k+6$. Hence, $\Pb_d^{(\tb)}$, and therefore $\Pb_d^{(\sb)}$, has at least two interior points, and is not Fano.
\end{proof}

We can go further to provide necessary and sufficient conditions for when $\Pb_d^{(\sb)}$ is translation equivalent to a reflexive polytope.

\begin{Theorem}\label{strictincreaseReflexive}
Suppose that $\sb$ is a sequence of strictly increasing positive integers such that $\Pb_d^{(\sb)}$ is Fano. Then $\Pb_d^{(\sb)}$ is  reflexive (up to translation) if and only if for each $0\leq i\leq d-1$, $k_i=s_{i+1}-s_i$ has the property $k_i | s_i$ and $k_i | s_{i+1}$.
\end{Theorem}
\begin{proof}
If  $\Pb_d^{(\sb)}$ is Fano, by Theorem \ref{strictincreaseFano} we know that the interior point is $(s_d-1,s_{d-1}-1,\ldots,s_2-1,s_1-1)^T$. If we translate $\Pb_d^{(\sb)}$ such that the interior point is the origin, the resulting polytope has vertices given by the columns of 
	\[
	\begin{bmatrix}
	1-s_d & 1 & 1 & 1 & \cdots & 1\\
	1-s_{d-1} & 1-s_{d-1} & 1 & 1 & \cdots & 1\\
	1-s_{d-2} & 1-s_{d-2} & 1-s_{d-2} & 1 & \cdots & 1\\
	\vdots & \vdots & \vdots & \ddots & & \vdots\\
	-1 & -1 & -1 & \cdots & -1 & 1\\
	\end{bmatrix}
	\]
This polytope has $\Hc$-representation
	\begin{itemize}
	\item $x_d\leq 1$
	\item $s_{i+1}x_{i}-s_{i}x_{i+1}\leq s_{i+1}-s_i$ for all $1\leq i\leq d-1$
	\item $-x_1\leq 1$
	\end{itemize}
using the convention of $x_d$ given by the first row and so on with $x_1$ given by the last row,  as it is clear that each vertex satisfies $d$ equations with equality and 1 with strict inequality.  

It follows then that $(\Pb_d^{(\sb)})^\vee$ is a lattice polytope if and only if $k_i | s_i$ and $k_i | s_{i+1}$ where $k_i=s_{i+1}-s_i$.			
\end{proof}

We have the following corollary.

\begin{Corollary}
Suppose $\sb$ is a sequence of strictly increasing positive integers. 
Then $\Pb_d^{(\sb)}$ is Gorenstein of index 2 if and only $\sb=(\frac{t_1}{2},\frac{t_2}{2},\ldots, \frac{t_d}{2})$ where $\tb=(t_1,\ldots,t_d)$ is a sequence such that $\Pb_d^{(\tb)}$ is reflexive.
Moreover, there is no sequence $\sb$ of strictly increasing positive integers such that $\Pb_d^{(\sb)}$ is Gorenstein of index $\geq 3$.
\end{Corollary}
\begin{proof}
This follows immediately from the observation that $r\Pb_d^{(\sb)}=\Pb_d^{(rs_1,rs_2,\ldots,rs_d)}$ and the condition that $s_1=2$ when $\Pb_d^{(\sb)}$ is reflexive.
\end{proof}

\subsection{Constant then strictly increasing $\sb$-sequences}
Suppose that we have a sequence of positive integers $\sb=(s_1,s_2,\ldots,s_i,s_{i+1},\ldots,s_d)$ such that $s_1=s_2=\cdots=s_i$ and $s_j<s_{j+1}$ for all $j\geq i$. We will given necessary and sufficient conditions for when $\Pb_d^{(\sb)}$  is translation equivalent to a Fano polytope for such sequences.

\begin{Theorem}\label{ConstantThenStrictFano}
Suppose that $\sb$ is a sequence such that  $s_1=\cdots=s_i$ for some $1 \leq i \leq d$ ,and $s_j<s_{j+1}$ for all $i\leq j\leq d-1$.  The polytope $\Pb_d^{(\sb)}$ is translation equivalent to a Fano polytope if and only if $s_1=\cdots=s_i=i+1$ and for all $j\geq i$, $s_{j+1}\leq 2s_j$. Moreover, the unique interior point is $(s_d-1,\ldots,s_{i+1}-1, i, i-1,\ldots,2,1)^T$.
\end{Theorem}

\begin{proof}
Suppose that $\sb$ is a sequence of this form such that $s_1=\cdots=s_i=i+1$ and $s_{j+1}\leq 2s_j$ for all $j\geq i$. We will show that $\delta_d(\Pb_d^{(\sb)})=1$ by showing that there is a unique inversion sequence $\eb$ such that $\asc\,\eb=d$. Let $\eb=(1,2,\ldots,i,s_{i+1}-1,s_{i+2}-1,\ldots,s_d-1)$. It is clear that this sequence has $d$ ascents, as $\frac{c}{i+1}<\frac{c+1}{i+1}$ for all $1\leq c\leq i$,
	\[
	\frac{s_j-1}{s_j}<\frac{s_{j+1}-1}{s_{j+1}}
	\] 
for all $j>i$ because $s_j<s_{j+1}$, and 
	\[
	\frac{i}{i+1}<\frac{s_{i+1}-1}{s_{i+1}}=1-\frac{1}{s_{i+1}}
	\]	
because $s_{i+1}>i+1$. To claim that this is the unique such inversion sequence, note that the only way to obtain an ascent each of the first $i$ positions is have the sequence begin $1,2,\ldots, i$. From previous work, we know that 
	\[
	\frac{s_j-1}{s_j}<\frac{s_{j+1}-2}{s_{j+1}}
	\]
cannot hold by the assumption $s_{j+1}\leq 2s_j$ for all $j\geq i$. This ensures that no other such inversion sequence with $d$ ascents exists. Thus, we have $\delta_d(\Pb_d^{(\sb)})=1$ so the polytope is Fano. Additionally, because we have 
	\[
	0<\frac{1}{i+1}<\cdots<\frac{i}{i+1}<\frac{s_{i+1}-1}{s_{i+1}}<\cdots<\frac{s_d-1}{s_d}<1
	\]
the point $(s_d-1,\ldots,s_{i+1}-1, i, i-1,\ldots,2,1)^T$ is in $\Pb_d^{(\sb)}$ and cannot lie on any supporting hyperplane and is hence the unique interior point.

Now, suppose that $\sb$ does not have the desired properties. We will show that $\Pb_d^{(\sb)}$ is not Fano. There are 3 possibilities:
	\begin{enumerate}[(i)]
	\item $s_1=\cdots=s_i\leq i$;
	\item $s_1=\cdots=s_i\geq i+2$;
	\item $s_1=\cdots=s_i=i+1$, but there exists some $j\geq i$ such that $2s_j<s_{j+1}$
	\end{enumerate}		
Each of these cases preclude $\Pb_d^{(\sb)}$ from being Fano. 

For (i), note that it is impossible for there to be an ascent in each of the first $i$ positions. Hence, we have $\delta_d(\Pb_d^{(\sb)})=0$.

For (ii), notice that $\Pb_d^{(i+2,\ldots,i+2,i+3,i+4,\ldots,d+2)}\subset \Pb_d^{(\sb)}$. If we consider inversion sequences in $\Ib_d^{(i+2,\ldots,i+2,i+3,i+4,\ldots,d+2)}$, we have that both $\eb=(1,2,\ldots,i,i+1,i+2,\ldots,d)$ $\eb'=(2,3,\ldots,i+1, i+2,i+3,\ldots,d+1)$ have the property $\asc\,\eb=\asc\,\eb'=d$ and hence $\delta_d(\Pb_d^{(i+2,\ldots,i+2,i+3,i+4,\ldots,d+2)})\geq 2$, which implies it has at least two interior points, which are also interior points of $\Pb_d^{(\sb)}$.

For (iii), note that $\Pb_d^{(2,3,\ldots,i+1,s_{i+1},\ldots,s_d)}\subset \Pb_d^{(\sb)}$. 
%%Old Proof of (iii)
%If we consider inversion sequences in $\Ib_d^{(i+1,\cdots,i+1,s_{i+1},\cdots, s_j, 2s_j+1,2s_j+2,\cdots, 2s_j+(d-j))}$, we have that both $\eb=(1,2,\cdots,i,s_{i+1}-1,\cdots, s_j-1,2s_j-1,2s_j,\cdots, 2s_j+(d-j-2))$ and $\eb'=(1,2,\cdots,i,s_{i+1}-1,\cdots, s_j-1,2s_j,2s_j+1,\cdots, 2s_j+(d-j-1))$ each have $\asc\,\eb=\asc\,\eb'=d$, where the arguments for this are identical to those in the previous section. Therefore, $\delta(\Pb_d^{(i+1,\cdots,i+1,s_{i+1},\cdots, s_j, 2s_j+1,2s_j+2,\cdots, 2s_j+(d-j))})\geq 2$ which means that it has at least two interior points, and by extension so must $\Pb_d^{(\sb)}$.
By the proof of Theorem \ref{strictincreaseFano}, we know that $\delta_d(\Pb_d^{(2,3,\ldots,i+1,s_{i+1},\ldots,s_d)})\geq 2$, which implies that %$\delta_d( \Pb_d^{(i+1,\cdots,i+1,s_{i+1},\cdots, s_j, 2s_j+1,2s_j+2,\cdots, 2s_j+(d-j))})\geq 2$ and thus 
$\delta_d(\Pb_d^{(\sb)})\geq 2$. 
\end{proof}

Now that we have a complete characterization of when $\Pb_d^{(\sb)}$ is Fano for $\sb$ of this type, we can now give necessary and sufficient conditions for when they are reflexive.

\begin{Theorem}\label{ConstantThenStrictReflexive}
Suppose that $\sb$ is a sequence such that  $s_1=\cdots=s_i$ for some $1 \leq i \leq d$ ,and $s_j<s_{j+1}$ for all $i\leq j\leq d-1$
and suppose that $\Pb_d^{(\sb)}$ is Fano. Then $\Pb_d^{(\sb)}$ is reflexive if and only if for all $i\leq j\leq d-1$ we have $k_j|s_j$ and $k_j|s_{j+1}$ where $k_j=s_{j+1}-s_j$.
\end{Theorem}

\begin{proof}
By Theorem \ref{ConstantThenStrictFano}, we know that the interior point is $(s_d-1,\ldots,s_{i+1}-1, i, i-1,\ldots,2,1)^T$. If we translate $\Pb_{d}^{(\sb)}$ so the interior point is the origin, the resulting polytope has vertices given as the columns of 
	\[
	\begin{bmatrix}
	1-s_d & 1 & 1 & \cdots & 1 & 1 & 1 & \cdots & 1 & 1\\
	1-s_{d-1} & 1-s_{d-1} & 1 & \cdots & 1 & 1 & 1 & \cdots & 1 & 1\\ 
	1-s_{d-2} & 1- s_{d-2} & 1-s_{d-2} & \cdots & 1 & 1 & 1 & \cdots & 1 & 1\\
	\vdots & \vdots &  & \ddots & \vdots & \vdots & \vdots & & \vdots & \vdots\\
	1-s_{i+1} & 1-s_{i+1} & 1-s_{i+1} & \cdots  & 1-s_{i+1} & 1 & 1 & \cdots & 1 & 1\\
	-i & -i & -i & \cdots & -i  & -i & 1 & \cdots & 1 & 1\\
	1-i & 1-i & 1-i & \cdots & 1-i & 1-i & 1-i & \cdots & 2 & 2\\
	\vdots & \vdots & \vdots & & \vdots & \vdots &  & \ddots & & \vdots\\
	-1	& -1 & -1 & \cdots & -1 & -1 & -1 & \cdots & -1 & i-1\\
	\end{bmatrix}.
	\]
This polytope has $\Hc$-representation
	\begin{itemize}
	\item $-x_1\leq 1$;
	\item $x_d\leq 1$;
	\item $x_{j-1}-x_j\leq 1$ for all $2\leq j\leq i$;
	\item $s_{j+1}x_j-s_jx_{j+1}\leq s_{j+1}-s_j$ for all $i\leq j\leq d-1$.
	\end{itemize}
using the convention that $x_d$ is given by the first row and so on with $x_1$ given by the last row. It is easy to see that each column of the matrix satisfies precisely $d$ equations with equality and 1 with strict inequality validating the $\Hc$-representation. It follows then that the dual polytope $(\Pb_d^{(\sb)})^\vee$ is a lattice polytope exactly when $k_j|s_j$ and $k_j|s_{j+1}$ where $k_j=s_{j+1}-s_j$ for $i\leq j\leq d-1$.	
\end{proof}

We can additionally give a description of Gorenstein lecture hall polytopes where $\sb$ is of this form.

\begin{Corollary}
Suppose that $\sb$ is a sequence such that  $s_1=\cdots=s_i$ for some $1 \leq i \leq d$ ,and $s_j<s_{j+1}$ for all $i\leq j\leq d-1$. 
Then $\Pb_d^{(\sb)}$ is Gorenstein of index $k\in\ZZ_{>0}$ if and only if there exists a sequence $\tb=(t_1,\ldots, t_d)$ such that $t_j=ks_j$ for all $j$ (which implies that $t_1=\cdots=t_i$ and $t_j<t_{j+1}$ for $j\geq i$) and $\Pb_d^{(\tb)}$ is reflexive.
\end{Corollary}
\begin{proof}
This is immediate with the observation that $k\Pb_d^{(\sb)}=\Pb_d^{(\tb)}$ and applying the conditions given in Theorem \ref{ConstantThenStrictReflexive}.
\end{proof}

\subsection{$\sb$-sequences increasing by at most 1}
We now consider an additional subclass of $\sb$-sequences. Suppose the $\sb=(s_1,s_2,\ldots,s_d)$ is a sequence of positive integers such that $s_i\leq s_{i+1}$ and $0\leq s_{i+1}-s_i\leq 1$ for all $1\leq i\leq d-1$. We have the following characterizations for when $\Pb_d^{(\sb)}$ is Fano and reflexive. 

\begin{Theorem}\label{atMost1Fano}
Suppose that $\sb=(s_1,s_2,\ldots,s_d)$ is a sequence of positive integers such that $s_i\leq s_{i+1}$ and $0\leq s_{i+1}-s_i\leq 1$ for all $1\leq i\leq d-1$. 
Then $\Pb_d^{(\sb)}$ is translation equivalent to a Fano polytope if and only if $s_d=d+1$. Moreover, the unique interior point is $(d,d-1,\ldots,2,1)^T$. 
\end{Theorem}

\begin{proof}
Suppose that $s_d=d+1$. We will show that there is a unique $\eb\in\Ib_d^{(\sb)}$ such that $\asc\,\eb=d$. It is clear that the sequence $\eb=(1,2,\ldots,d)$ satisfies this property, as both $\frac{i}{k}<\frac{i+1}{k}$ and $\frac{i}{k}<\frac{i+1}{k+1}$ are true which implies $\frac{i}{s_i}<\frac{i+1}{s_{i+1}}$. Moreover, to have maximum ascents, we must have $e_i<e_{i+1}$, which means that if $e_d\leq d-1$, $e_1=0$ implying that there is no ascent in the first position. Thus, the sequence $\eb=(1,2,\ldots, d)$ is the only inversion sequence with $d$ ascents, giving $\delta_d(\Pb_d^{(\sb)})=1$. It also follows that the unique interior point of $\Pb_d^{(\sb)}$ is $(d,d-1,\ldots,2,1)^T$, as 
	\[
	0<\frac{1}{s_1}<\frac{2}{s_2}<\cdots<\frac{d}{s_d}<1
	\]
implies that the point is in $\Pb_d^{(\sb)}$ and not on any supporting hyperplane. 	

Now, note that if $s_d\geq d+2$, both the inversion sequences $(1,2,3,\ldots,d)$ and $(2,3,\ldots,d+1)$ has $d$ ascents. Thus, $\delta_d(\Pb_d^{(\sb)})\geq 2$ in this case.

If we have that $s_d\leq d$, it follows that $\Pb_d^{(\sb)}\subseteq \Pb_d^{(\tb)}$ where $\tb=(d,d,\ldots,d)$. Since it is clear that for $\eb\in\Ib_d^{(\tb)}$  we have $i\in \Asc\,\eb$ if and only if $e_{i-1}<e_i$ and since $e_i\in\{0,1,...,d-1\}$ there is no sequence with $\asc\,\eb=d$. Thus, we have $\delta_d(\Pb_d^{(\sb)})=\delta_d(\Pb_d^{(\tb)})=0$.
\end{proof}

\begin{Theorem}
Suppose that $\sb=(s_1,s_2,\ldots,s_d)$ is a sequence of positive integers such that $s_i\leq s_{i+1}$ and $0\leq s_{i+1}-s_i\leq 1$ for all $1\leq i\leq d-1$.	
and suppose that $\Pb_d^{(\sb)}$ is Fano. Then $\Pb_d^{(\sb)}$ is reflexive if and only if $k_i|s_i$ and $k_i|s_{i+1}$ where $k_i=(i+1)s_i-is_{i+1}$. 
\end{Theorem}

\begin{proof}
By Theorem \ref{atMost1Fano}, the interior point of $\Pb_d^{(\sb)}$ is $(d,d-1,\ldots,2,1)^T$. So, if we translate the polytope such that the origin is the interior point, we have the polytope with vertices
	\[
	\begin{bmatrix}
	-d & 1 & 1 & 1 & \cdots & 1\\ 
	1-d & 1-d & s_{d-1}-d+1 &s_{d-1}-d+1 & \cdots & s_{d-1}-d+1\\
	2-d & 2-d & 2-d & s_{d-2}-d+2 & \cdots &  s_{d-2}-d+2 \\
	\vdots & \vdots & \vdots & & \ddots & \cdots\\
	-1 & -1 & -1 & \cdots & -1 & s_1-1\\
	\end{bmatrix}
	\]
which, using the convention of $x_d$ given by the first row and so on with $x_1$ given by the last row, has the $\mathcal{H}$-representation 
	\begin{itemize}
	\item $x_d\leq 1$
	\item $s_{i+1}x_i-s_ix_{i+1}\leq (i+1)s_i-is_{i+1}$ for all $1\leq i\leq d-1$
	\item $-x_1\leq 1$
	\end{itemize}		
as it is not hard to see that each vertex satisfies $d$ equations with equality and 1 equation with strict inequality. It is now clear that $(\Pb_d^{(\sb)})^\vee$ is a lattice polytope if and only if $k_i|s_i$ and $k_i|s_{i+1}$ for $k_i=(i+1)s_i-is_{i+1}$.  
\end{proof}

\section{Integral decomposition property and triangulations}

We say $\Pc$ satisfies the \emph{integral decomposition property (IDP)} if for all $\zb\in k\Pc\cap\ZZ^d$ there exists $\xb_1,\xb_2,\ldots,\xb_k\in \Pc\cap \ZZ^d$ such that 
	\[
	\xb_1+\xb_2+\cdots+\xb_k=\zb.
	\]		 
If $\Pc$ satisfies then integer decomposition property, we say that $\Pc$ is IDP. For $\sb$-lecture hall polytopes where $\sb$ is monotonic sequence, we have the following theorem.

\begin{Theorem}\label{normality}
Let $\sb=(s_1,s_2,\ldots,s_d)$ be a monotone sequence of positive integers. Then the polytope $\Pb_d^{(\sb)}$ is IDP.  
\end{Theorem}
\begin{proof}
Without loss of generality, suppose that $\sb$ is increasing. We will show that given $k\geq 2$, for any $\xb\in k\Pb_d^{(\sb)}\cap \ZZ^d$, there exists some $\yb\in \Pb_d^{(\sb)}\cap \ZZ^d$ such that $(\xb-\yb)\in (k-1)\Pb_d^{(\sb)}\cap \ZZ^d$. Note that this is sufficient, because this result allows integral closure to follow from induction on $k$.

 First note that $k\Pb_d^{(\sb)}=\Pb_d^{(ks_1,ks_2,\ldots,ks_d)}$, which is clear by definition. Let $\xb=(x_d,x_{d-1},\ldots,x_1)^T\in k\Pb_d^{(\sb)}\cap \ZZ^d$, so we have that $\xb$ satisfies 
 	\[
 	0\leq \frac{x_1}{ks_1}\leq\frac{x_2}{ks_2}\leq \cdots\leq \frac{x_d}{ks_d}\leq 1.
 	\] 
Note that since $\sb$ is increasing, given any $C\in\ZZ_{>0}$ by the above we must have that $x_i\leq Cs_i$ implies that $x_{i-1}\leq Cs_{i-1}$ and likewise $x_i>Cs_i$ implies $x_{i+1}>Cs_{i+1}$. So, let $1\leq j\leq d$ be the minimum index such that $x_j>(k-1)s_j$. Then we let 
	\[
	\yb=\left(x_d-(k-1)s_d, \ldots, x_j-(k-1)s_j,0,\ldots,0\right)^T
	\]
with $\yb={\bf 0}$ if there is no such $j$.
	
We know that the lattice point  is in $\Pb_d^{(\sb)}$ because for any $j\leq i<d$ we have 
	\[
 \frac{x_i-(k-1)s_i}{s_i}\leq \frac{x_{i+1}-(k-1)s_{i+1}}{s_{i+1}}
	\]
	 is equivalent to $\displaystyle \frac{x_i}{ks_i}\leq\frac{x_{i+1}}{ks_{i+1}}$ and $0<x_i-(k-1)s_i\leq s_i$ by construction.
	 
It is left to verify that $(\xb-\yb)=((k-1)s_d,\ldots,(k-1)s_j,x_{j-1},\ldots,x_1)^T \in \Pb_d^{((k-1)s_1,\ldots,(k-1)s_d)}\cap \ZZ^d$. However, this is immediate, because $\displaystyle \frac{x_i}{(k-1)s_i}\leq \frac{x_{i+1}}{(k-1)s_{i+1}}$ is equivalent to $\displaystyle \frac{x_i}{ks_i}\leq \frac{x_{i+1}}{ks_{i+1}}$ and it is clear that since $x_{j-1}\leq (k-1)s_{j-1}$ by assumption that 
	\[
	\frac{x_{j-1}}{(k-1)s_{j-1}}\leq\frac{(k-1)s_j}{(k-1)s_j}=1.
	\]	 
Thus, we have the  $\Pb_d^{(\sb)}$ is IDP.
\end{proof}

Recall that a \emph{triangulation} of a lattice polytope $\Pc$ is a subdivison of $\Pc$ into $d$-dimensional simplices. We say that a triangulation is \emph{unimodular} if each simplex $\Delta$  of the triangulation is unimodularly equivalent to the standard $d$-simplex or equivalently, each simplex has smallest possible normalized volume $\Vol(\Delta)=1$. One should note that a polytope $\Pc$ possessing a unimodular triangulation means that $\Pc$ can be covered by IDP polytopes which implies that $\Pc$ is IDP. We will show the existence for a unimodular triangulation of $\Pb_d^{(\sb)}$ provided that for all $1\leq i\leq d-1$, $s_{i+1}=n_is_i$ where $n_i\in\ZZ_{>0}$.

First, we define chimney polytopes. Given a polytope $\Pc\subset \RR^d$ and two integral linear functionals $\ell$ and $u$ such that $\ell\leq u$, then the \emph{chimney polytope} associated to $\Pc$, $\ell$, and $u$ is 
	\[
	\Chim(\Pc,\ell,u):=\{(\xb,y)\in\RR^d\times\RR \, : \, \xb\in\Pc, \, \ell(\xb)\leq y\leq u(\xb)\}.
	\]

For chimney polytopes we have the following theorem regarding triangulations.

\begin{Lemma}[{\cite[Theorem 2.8]{UnimodularTriangulations}}]\label{chimneytriangulation}
If $\Pc$ admits a unimodular triangulation, then so does $\Chim(\Pc,\ell,u)$.
\end{Lemma}	

With this in mind, we can now state and prove a theorem for $\Pb_d^{(\sb)}$ where $\sb$ is increasing of a particular form.

\begin{Theorem}\label{triangulation}
Let $\sb$ be an increasing sequence of positive integers such that $s_{i+1}=k_is_i$ for some $k_i\in\ZZ_{>0}$ for all $1\leq i\leq d-1$. Then $\Pb_d^{(\sb)}$ admits a unimodular triangulation.  
\end{Theorem}

\begin{proof}
Note that if $\sb$ has the property $s_d=k_{d-1}s_{d-1}$ for some $k_{d-1}\in\ZZ_{>0}$, we can express $\Pb_d^{(\sb)}$ as a chimney polytope, namely 
	\[
	\Pb_d^{(\sb)}\iso\Chim(\Pb_{d-1}^{(s_1,\ldots,s_{d-1})},k_{d-1}x_{d-1}, s_d)
	\]
where ${\bf s_d}$ is constant function of value $s_d$. It is easy to see this isomorphism as all of the supporting hyperplanes for 	$\Chim(\Pb_{d-1}^{(s_1,\ldots,s_{d-1})},k_{d-1}x_{d-1}, s_d)$ are those of $\Pb_{d-1}^{(s_1,\ldots,s_{d-1})}$ with the addition of $x_d\leq s_d$ and $k_{d-1}x_{d-1}\leq x_d$. However, these hyperplanes are precisely the supporting hyperplanes of $\Pb_d^{(\sb)}$.

Now, note that any 1 dimensional lecture hall polytope trivially has a unimodular triangulation. So, if $\sb$ has the property that $s_{i+1}=k_is_i$ for a positive integer $k_i$ for each $i$, then applying Theorem \ref{chimneytriangulation} to this inductive chimney polytope construction of $\Pb_d^{(\sb)}$ yields the existence of a unimodular triangulation.  
\end{proof}

\begin{Remark}
We should note that Theorem \ref{triangulation} implies that $\Pb_d^{(\sb)}$ where $\sb$ has the property $s_{i+1}=\frac{s_i}{k_i}$ for some positive integer $k_i$ for all $i$ also admits a unimodular triangulation. 
\end{Remark}

\section{Constructing new examples}
In this section, we construct new Gorenstein and IDP  lecture hall polytopes. We will do this by identifying an $\sb$-lecture hall polytope as the free sum of two smaller lecture hall polytopes which are Gorenstein and/or IDP. 

Recall that given two lattice polytopes $\Pc\subset \RR^{d_{\Pc}}$ and $\Qc\subset \RR^{d_{\Qc}}$ such that $0_{d_{\Pc}}\in \Pc$ and $0_{d_{\Qc}}\in \Qc$, the \emph{free sum} of $\Pc$ and $\Qc$ is the $(d_{\Pc}+d_{\Qc})$-dimensional polytope given by $\Pc\oplus\Qc=\con\{(0_{\Pc}\times \Qc)\cup(\Pc\times 0_{\Qc})\}$. We can view lecture hall polytopes as free sum of smaller lecture hall polytopes.

\begin{Proposition}\label{FreeSumDecomp}
For integer sequences $\sb=(s_1,\ldots,s_d)$ and $\tb=(t_1,\ldots,t_e)$, we have $\Pb_{d+e}^{(\sb,\tb)}\iso \Pb_d^{(\sb)}\oplus \Pb_e^{\left(\tilde{\tb}\right)}$, where $(\sb,\tb)=(s_1,\ldots,s_d,t_1,\ldots,t_e)$ and $\tilde{\tb}=(t_d,t_{d-1},\ldots,t_1)$. 
\end{Proposition}
\begin{proof}
Translate by the vector $(t_e,\ldots,t_2,t_1,0,0,\ldots,0)^T$. 
\end{proof}

The following generalization of Braun's formula gives us conditions on the $\delta$-polynomial of a free sum of two polytopes. 

\begin{Lemma}[{\cite[Theorem 1.4]{BeckJayawantMcAllister}}]\label{FreeSumDelta}
Let $\Pc\subset \RR^d$ and $\Qc\subset\RR^e$ be integral convex polytopes each containing its respective origin. Then $\delta(\Pc\oplus\Qc,\lambda)=\delta(\Pc,\lambda)\delta(\Qc,\lambda)$ holds if and only if either $\Pc$ or $\Qc$ satisfies that the equation of each facet is of the form $\sum_{i=1}^fa_ix_i=b$ where $a_i$ is an integer, $b\in\{0,1\}$, and $f\in\{d,e\}$.  
\end{Lemma}

We can now give a construction for larger lecture hall polytopes which must be Gorenstein.

\begin{Theorem}
Suppose that $\sb=(s_1,s_2,\ldots,s_d)$ and $\tb=(t_1,t_2,\ldots, t_e)$ are integer sequences such that $\Pb_d^{(\sb)}$ is Gorenstein of index $k$ and $\Pb_e^{(\tb)}$ is Gorenstein of index $\ell$.
Then $\Pb_{d+e+1}^{(\sb,1,\tb)}$ is Gorenstein of index $k+\ell$. 
\end{Theorem}
\begin{proof}
Note that by Proposition \ref{FreeSumDecomp}, we have that $\Pb_{d+e+1}^{(\sb,1,\tb)}\iso\Pb_{d+1}^{(\sb,1)}\oplus\Pb_e^{(\tilde{\tb})}$. By the $\Hc$-representation, we know that $\Pb_{d+1}^{(\sb,1)}$ satisfies that the equation of each facet is of the form $\sum_{i=1}^{d+1}a_ix_i=b$ where $a_i$ is an integer, $b\in\{0,1\}$. Moreover, from the $\mathcal{V}$-represntation it is clear that $\Pb_{d+1}^{(\sb,1)}\iso\Pyr(\Pb_{d}^{(\sb)})$, so it has the same $\delta$-vector and is thus Gorenstein. By Lemma \ref{FreeSumDelta}, we then know that $\delta(\Pb_{d+e+1}^{(\sb,1,\tb)},\lambda)=\delta(\Pb_d^{(\sb)},\lambda)\delta(\Pb_e^{(\tb)},\lambda)$ because $\Pb_e^{(\tilde{t})}\cong\Pb_e^{({t})}$. Therefore, $\delta(\Pb_{d+e+1}^{(\sb,1,\tb)},\lambda)$ is symmetric polynomial of degree $(d+e+1)-(k+\ell)+1$ and we have the desired. 
\end{proof}

Additionally, necessary and sufficient conditions for the integral closure of a free sum of two polytopes are known. These are given in the following theorem.

\begin{Lemma}[{\cite[Theorem 0.1]{HibiHigashitaniIDP}}]\label{FreeSumIDP}
Let $\Pc\subset \RR^d$ and $\Qc\subset\RR^e$ be integral convex polytopes each containing its respective origin. Suppose that $\Pc$ and $\Qc$ satisfy $\ZZ(\Pc\cap \ZZ^d)=\ZZ^d$, $\ZZ(\Qc\cap \ZZ^e)=\ZZ^e$, and
	\[
	(\Pc\oplus\Qc)\cap\ZZ^{d+e}=\mu(\Pc\cap \ZZ^d)\cup \nu(\Qc\cap \ZZ^e)
	\]
where $\mu$ and $\nu$ are the canonical injections defined $\mu:\RR^d\to \RR^{d+e}$ by $\alpha\mapsto(\alpha,0_e)$ and $\nu:\RR^e\to\RR^{d+e}$ by $\beta\mapsto(0_d,\beta)$. Then the free sum $\Pc\oplus\Qc$ is IDP if and only if the following two conditions hold:
	\begin{itemize}
	\item each of $\Pc$ and $\Qc$ is IDP;
	\item either $\Pc$ or $\Qc$ has the property that the equation of each facet is of the form $\sum_{i=1}^fa_ix_i=b$ where $a_i$ is an integer, $b\in\{0,1\}$, and $f\in\{d,e\}$.
	\end{itemize}	 	
\end{Lemma}

We can now give a construction for larger IDP lecture hall polytopes. 

\begin{Theorem}\label{FreeSumIDPLecture}
Suppose that $\sb=(s_1,s_2,\ldots,s_d)$ and $\tb=(t_1,t_2,\ldots, t_e)$ are integer sequences such that $\Pb_d^{(\sb)}$ and $\Pb_e^{(\tb)}$ are IDP.
Then $\Pb_{d+e+1}^{(\sb,1,\tb)}$ is IDP. 
\end{Theorem}
\begin{proof}
Note that for any 2 lecture hall polytopes  $\Pb_d^{(\sb)}$ and $\Pb_e^{(\tb)}$, we have $\ZZ(\Pb_d^{(\sb)}\cap \ZZ^d)=\ZZ^d$ and  $\ZZ(\Pb_e^{(\tb)}\cap \ZZ^e)=\ZZ^e$ follow immediately.
	
Now, by Proposition \ref{FreeSumDecomp}, we have that $\Pb_{d+e+1}^{(\sb,1,\tb)}\iso\Pb_{d+1}^{(\sb,1)}\oplus\Pb_e^{(\tilde{\tb})}$. By the $\Hc$-representation, we know that $\Pb_{d+1}^{(\sb,1)}$ satisfies that the equation of each facet is of the form $\sum_{i=1}^{d+1}a_ix_i=b$ where $a_i$ is an integer, $b\in\{0,1\}$. To see that
	\[ 
	(\Pb_{d+1}^{(\sb,1)}\oplus\Pb_e^{(\tilde{\tb})})\cap\ZZ^{d+1+e}=\mu(\Pb_{d+1}^{(\sb,1)}\cap \ZZ^{d+1})\cup \nu(\Pb_e^{(\tilde{\tb})}\cap \ZZ^e)
	\]	
holds, note that the right side is clearly contained in the left side. If we consider an element $\xb$ such that 
	\[
	\xb\in(\Pb_{d+1}^{(\sb,1)}\oplus\Pb_e^{(\tilde{\tb})})\cap\ZZ^{d+1+e}\setminus\left(\mu(\Pb_{d+1}^{(\sb,1)}\cap \ZZ^{d+1})\cup \nu(\Pb_e^{(\tilde{\tb})}\cap \ZZ^e)\right),
	\]
we have that $\sum_{i=1}^{d+e+1}c_i v_i=1$ where $c_i$ is constant and $v_i$ is the $i$th vertex. However, we also must have that $x_{d+1}=1$, which implies that $\sum_{i=1}^{d+1}c_i=1$ from the definition of the free sum. So this implies that $\xb\in \mu (\Pb_{d+1}^{(\sb,1)}\cap \ZZ^{d+1})$ which is a contradiction. The result now follows from Lemma \ref{FreeSumIDP}.
\end{proof}

\section{Concluding Remarks}
While we have been able to ascertain many previously unknown properties of lecture hall polytopes, full characterizations of all of these properties remain elusive. We conclude with two conjectures.

\begin{Conjecture}\label{IDPconj}
For any $\sb=(s_1,\cdots,s_d)$, $\Pb_d^{(\sb)}$ is IDP.
\end{Conjecture}
For many randomly generated $\sb$, we have found $\Pb_d^{(\sb)}$ to be IDP and we have been unable to find an example of a non IDP lecture hall polytope. Additionally, the convenient description of dilates of lecture hall polytopes, namely $c\Pb_d^{(\sb)}=\Pb_d^{(cs_1,cs_2,\cdots,cs_d)}$, suggests that one may be able to generalize our arguments for monotone sequences to arbitrary $\sb$.

\begin{Conjecture}
For any $\sb=(s_1,\cdots,s_d)$, $\Pb_d^{(\sb)}$ admits a unimodular triangulation.
\end{Conjecture}
We have come across no examples of lecture hall polytopes which do not admit a unimodular triangulation. However, using Gr\"obner bases has not proved fruitful given that though a variable ordering and monomial ordering which yield a quadratic squarefree Gr\"obner basis seem to always exist, it is not always the same ordering. A positive answer to this conjecture would resolve Conjecture \ref{IDPconj} as well. Moreover, a counterexample, or a positive partial result such as the monotone case would be of great interest.   

%\begin{Question}
%Are there (nice) necessary and sufficient conditions for when $\Pb_d^{(\sb)}$ is Fano, reflexive, or Gorenstein, provided that $\sb$ is a monotone sequence?
%\end{Question}
%While we have covered particular cases, it would be nice to completely understand the monotone situation.

\end{document}